\documentclass[12pt]{article} 
\usepackage{amssymb,amsmath,amsthm}
\usepackage{graphicx}
\usepackage{hyperref}

\newcommand{\RR}{\mathbb{R}}

\newcommand{\NN}{\mathbb{N}}

\newcommand{\xx}{\mathbf{x}}

\newcommand{\uu}{\mathbf{u}}

\newcommand{\eps}{\varepsilon}
\newcommand{\const}{\mathrm{const}}
\newcommand{\Tra}[1]{\mathcal{T}_{#1}}

\newcommand{\norm}[1]{\langle #1\rangle} 

\newcommand{\GDM}{GDM}

\newcommand{\beq}[1]{\begin{equation}\label{#1}}
\newcommand{\eeq}{\end{equation}}

\newcommand{\ba}[1]{\begin{array}{#1}}
\newcommand{\ea}{\end{array}}
\newcommand{\mat}[2]{\left[\ba{#1}#2\ea\right]}

\newcommand{\dst}{\displaystyle}

\newcommand{\myname}{serge.sadov} 

\newtheorem{theorem}{Theorem}
\newtheorem{prop}{Proposition}
\newtheorem{lemma}{Lemma}

\title{On Shallit's minimization problem}
\author{Sergey Sadov
\footnote{E-mail: {\myname}@gmail.com}}

\begin{document}
\maketitle

\begin{abstract}
We revisit J.~Shallit's minimization problem from 1994 SIAM Review concerning a two-term asymptotics of the minimum of a certain rational 
sum involving variables and products of their reciprocals, the number of variables being the large parameter.
Properties previously known numerically, most importantly, the existence of the constant in the asymptotics, are proved.
We supply a sharp remainder estimate to the originally proposed asymptiotic formula. 
The proofs are based on the analysis of trajectories of a planar discrete dynamical system that determines the point of minimum. 

Keywords: analytical inequalities, Shallit's constant, 
hyperbolic point, local linearization, convergence rate.

MSC primary: 
26D15,  
secondary:
37N40,  
41A25.  
\end{abstract}

\section{Introduction}

The following minimization problem was proposed in \cite{Shallit_1994}. 

\medskip
{\it
Let $\xx=(x_1,\dots,x_n)$ be a vector with positive components.
Denote
$$
 f_n(\xx)=\sum_{i=1}^n x_i+\sum_{1\le i\le j\le n} \prod_{k=i}^j\frac{1}{x_k}.
$$
Show that there exists a positive constant $C$ such that
\begin{equation}
\label{ShAsym}
 \min_{\xx>0} f_n(\xx)=3n-C+o(1) 
 \qquad (n\to\infty).
\end{equation}
}

Hereinafter we write $\xx>0$ or $\xx\in\RR_+^n$ if $x_i>0$ for $i=1,\dots,n$.

The constant $C=1.3694514039\dots$ is nowadays known as Shallit's constant \cite[Sec.\ 3.1]{Finch2003}. 

Denote
\begin{equation}
\label{An}
 A_n=\inf_{\xx>0} f_n(\xx).
\end{equation}
and
\begin{equation}
\label{Cn}
 C_n=3n-A_n.
\end{equation}
The publication \cite{GDM_1995} in the Solutions section of SIAM Review, to which we refer in the sequel as \GDM, 
does not, as a matter of fact, answer the original question proper. It addresses the method of computation of the constant $C$
and contains a series of observations,  supported by a persuasive numerical evidence, about the sequence $(C_n)$ and certain auxiliary sequences. 
However the convergence is not proved. The goal of this paper is to provide the proof and to justifiy analytically the various claims made in \GDM. 
We are able to replace the $o(1)$ remainder in \eqref{ShAsym} by $\Theta(\rho^{-n})$, where $\rho=2+\sqrt{3}$.
The notation $f=\Theta(g)$ or $f\asymp g$ is used as a shortcut to the two-sided inequality $m|f|\leq |g|\leq M|g|$ with
some constants $0<m<M$. 

In Section~\ref{sec:prelim}, following \GDM\ and just being more explicit about the mechanism underlying the asymptotics in question, we 
bring to the forefront the two-dimensional dynamical system 
whose special trajectories correspond to the points of minimum in Shallit's problem for different values of $n$. 
A sequence of short lemmas covering various properties of the trajectories, mostly known numerically from \GDM,
comprises Section~\ref{sec:trajan}. This preparation makes it easy to prove convergence of the sequence $(C_n)$ to a finite limit (Section~\ref{sec:asmin},
Theorem~\ref{thm:Sh-lim}).
In a somewhat more difficult 
Section~\ref{sec:convrates} we study the precise rates of convergence of $(C_n)$ and other relevant sequences. 
Note that part (c) of Theorem~\ref{thm:convrates} is a refinement of Theorem~\ref{thm:Sh-lim}.

The dynamical system considered here is area-preserving and can be described in terms of the least action principle for an appropriate Lagrange's function.
A study of its global behaviour may be of interest in its own right.  We make some remarks to that effect in the short Section~\ref{sec:remarks}.

Appendix
contains the numerical values of the constants $C$ and $p_0^*$ (defined in Eq.~\ref{limits-pu}) to the accuracy of 400 digits,
which can help those readers who wish to play with numbers.
 
In \cite{Finch2003} Shallit's constant neighbors the Shapiro-Drinfeld constant, which is related
to the problem I studied in \cite{Sadov_2016}. This combination 
is not altogether incidental, but the connection is too indirect to attempt to describe it precisely.

For the reader's convenience the notation of \GDM\ is 
used where applicable and new notation 
 is chosen so as  not to interfere with that of \GDM. 
There is one exception: we denote by $\lambda_j^*$ what GDM denote $\lambda_j$
 (see Section~\ref{ssec:pu-limits}, Remark 2).

\section{Preliminaries}
\label{sec:prelim}

\subsection{Reduction and critical point equations} 
\label{ssec:reduction}
Following \GDM, recall a reduction of the minimization problem to the analysis of solution of a system of algebraic equations. The substitution
\beq{chvar}
 x_1=\frac{1}{u_1}; \qquad x_j=\frac{1+u_{j-1}}{u_j}, \quad j=2,3,\dots,n,
\eeq
leads to the identity
$$
 f_n(\xx)=g_n(\uu),
$$
where the function $g_n$ is defined by the formula involving $O(n)$ summands,
\begin{equation}
\label{gn}
 g_n(\uu)=\sum_{j=1}^n L(u_{j-1},u_j),
 \qquad
 L(t,s)=s+\frac{1+t}{s},
\end{equation}
and $\uu=(0,u_1,u_2,\dots,u_n)$. Put $\hat\uu=(u_1,u_2,\dots,u_n)$.
The transformation \eqref{chvar}, $\xx\mapsto\hat\uu$, is a bijection of $\RR_+^n$ onto itself. Hence
\begin{equation}
\label{inf-f-g}
 A_n=\inf_{\hat\uu>0} g_n(\uu).
\end{equation}
Here `$\inf$' can be replaced by `$\min$' ---this is almost obvious and taken for granted in 
\cite{Shallit_1994}
and in \GDM. 
Let us give a formal proof: since $g_n(0,1,1,\dots,1)=3n-1$, 
it suffices in the right-hand side of \eqref{inf-f-g} to consider $\hat\uu\in K$, where  $K$ is the cube in $\RR_+^n$ defined by the inequalities
$ (3n-1)^{-1}\leq u_j \leq 3n-1$, $j=1,\dots,n$.
Since $K$ is compact, the function $g_n$ attains its minimum value $A_n\leq 3n-1$ on $K$; hence $C_n\geq 1$ by \eqref{Cn}. 

The necessary condition of extremum, $\nabla g_n(\uu)=0$, holds at a point of minimum.  
In the explicit form it reads
\begin{equation}
\label{criteq}
\ba{l}\dst
\frac{1+u_{j-1}}{u_j^2}= \frac{1}{u_{j+1}}+1, \qquad j=1,\dots,n-1,
  \\[3ex]\dst 
  \frac{1+u_{n-1}}{u_n^2}=1,
  \quad u_0=0.
\ea
\end{equation}
We prove below, in Lemma~\ref{lem:uniq}(b), that the solution is unique. 
For this reason we write ``{\it the}\ solution $\uu$'', ``{\it the}\ trajectory $\Tra{n}$''
from now on, although, logically, the definite article is not fully justified until Lemma~\ref{lem:uniq}(b).  

\subsection{The dynamical system} 
\label{ssec:dynsys}

We will interpret the critical point equations \eqref{criteq} as a boundary value problem for a trajectory of a dynamical system
with discrete time. 
Introduce the partial self-map of $\RR^2$, $\Phi:(p,u)\mapsto (p',u')$, by the formulas
\begin{equation}
\label{mapPhi}
 p'=p^2(u+1)-1,\qquad u'=\frac{1}{p} \quad (p\neq 0).
\end{equation}
A \emph{trajectory}\ $\Tra{}$ with initial point $(p_0,u_0)$ is 
a finite or infinite sequence of  iterations $(p_j,u_j)=\Phi^{(j)}(p_0,u_0)$.
In general, one may consider trajectories in $\RR^2$, but we will only need the points with $p_j\geq 0$. 
 
Specifically, let  $\uu=\uu^{(n)}$ be the solution vector of the critical point equations \eqref{criteq}.
Put $p_{j-1}=1/u_j$ ($j=1,\dots,n$) and $p_n=0$.
The finite sequence $\Tra{n}=\{(p_j,u_j)\}$, $j=0,1,\dots,n$, is a trajectory of the map $\Phi$ such that 
$p_j>0$ for $0\leq j\leq n-1$ and the boundary conditions
\begin{equation}
\label{bcond}
u_0=p_n=0
\end{equation}
are satisfied. 
%

\section{Analysis of trajectories}
\label{sec:trajan}

\subsection{A qualitative overview}
\label{ssec:overview}

The unique fixed point of the map $\Phi$ in $\RR_+^2$ is $P_0=(1,1)$.
It is hyperbolic: the Jacobi matrix 
$$
 D\Phi(1,1)=\mat{cc}{4 & 1\\-1 & 0}
$$
has the eigenvalues $\rho=2+\sqrt{3}$ and $\rho^{-1}=2-\sqrt{3}$. From the general theory of dynamical system it is known 
(see e.g.\ \cite[Theorem~6.2.3]{KH_1996}) that there exist the stable curve $\gamma_s$ and the unstable curve
$\gamma_u$ defined at least in some neighborhooud of the point $P_0$. 
The curve $\gamma_s$ is tangent to the stable separatrix $\tau_s$ of the linearized map $D\Phi(1,1)$
and $\gamma_u$ is tangent to the unstable separatrix $\tau_u$.  
  
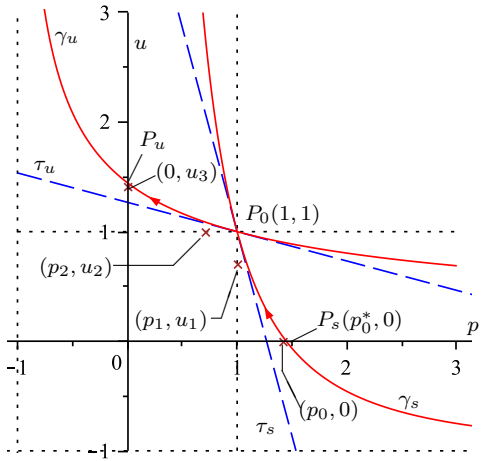
\begin{figure}[ht]
\begin{center}
{\scriptsize 
\begin{picture}(175,175)
\put(0,0){\includegraphics[viewport=10 81 360 426, clip=true,scale=0.5]{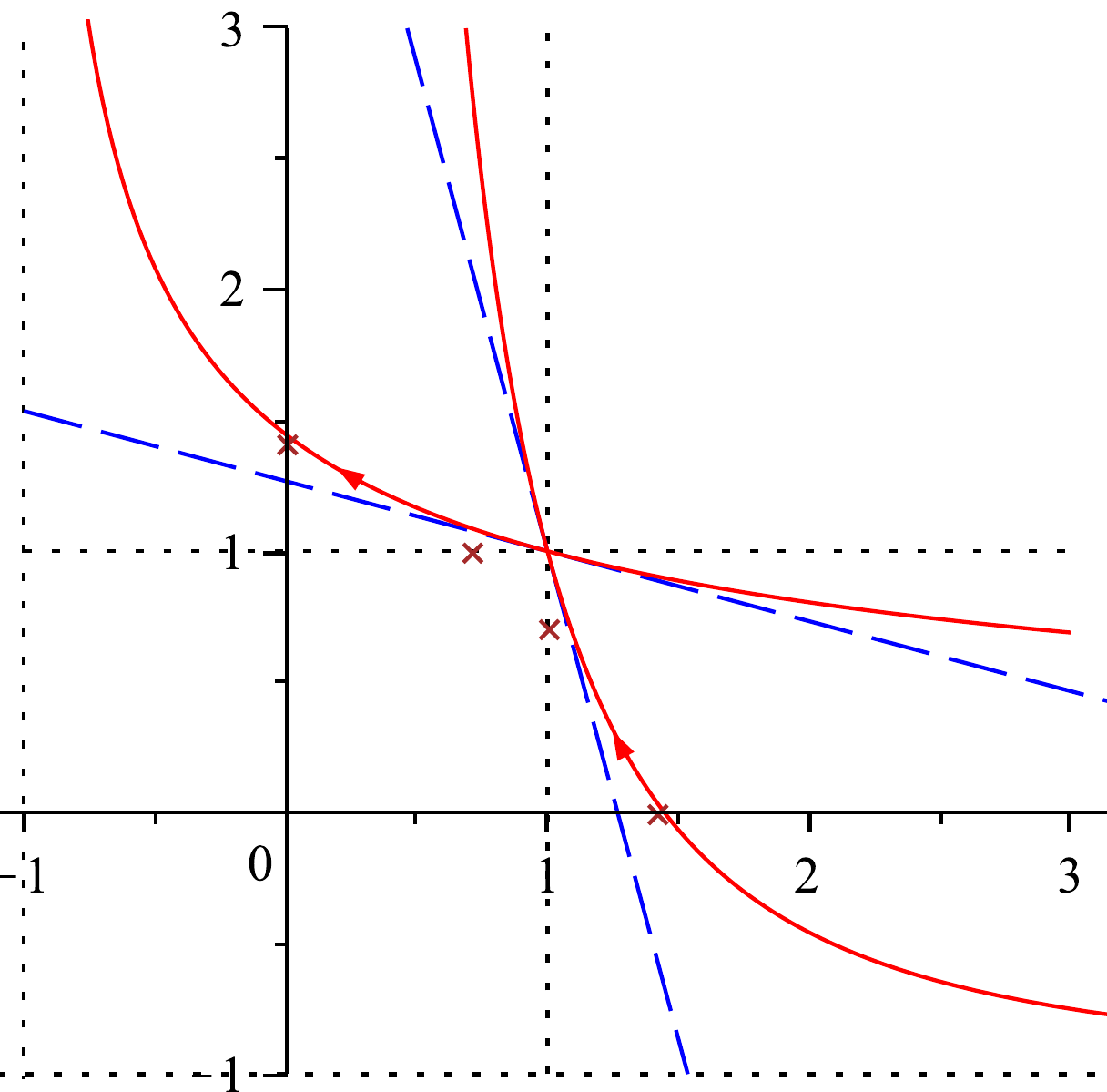}}
\put(174,48){$p$}
\put(48,155){$u$}
\put(94,10){$\tau_s$}
\put(10,109){$\tau_u$}
\put(148,20){$\gamma_s$}
\put(18,158){$\gamma_u$}
%
\put(90,90){$P_0(1,1)$}
\thinlines
\put(107,44){$\line(2,1){10}$}
\put(117,50){$P_s(p_0^*,0)$}
\put(46,105){$\line(1,2){5}$}
\put(50,118){$P_u$}
\put(104,43){$\line(0,-1){10}$}
\put(104,33){$\line(1,-2){7}$}
\put(110,16){$(p_{0},0)$}
\put(87,71){$\line(-1,-3){6}$}
\put(81,53){$\line(-1,0){4}$}
\put(48,50){$(p_{1},u_{1})$}
\put(73,83){$\line(-1,-1){10}$}
\put(63,73){$\line(-1,0){23}$}
\put(12,70){$(p_{2},u_{2})$}
\put(46,102){$\line(2,1){10}$}
\put(56,106){$(0,u_{3})$}
\end{picture}
}
\end{center}
\caption{Invariant curves and the trajectory $\Tra{3}$}  
\label{fig:invcurves}
\end{figure}
   
For the existence of the asymptotics \eqref{ShAsym} the following fact is crusial: %
the stable curve intersects the semiaxis $u=0$, $p>0$ at some point $P_s$;
similarly, the unstable curve intersects the semiaxis $p=0$, $u>0$ at some point $P_u$.
The trajectory $\Tra{n}$ satisfying the boundary conditions \eqref{bcond} begins near $P_s$ and ends near $P_u$.
Earlier iterations rapidly approach $P_0$ almost along the arc of $\gamma_s$ between $P_s$ and $P_0$, while later iterations move away from the fixed point almost along the arc of $\gamma_u$ between $P_0$ and $P_u$. 
All but $O(1)$ of the total number $n+1$ points of the trajectory lie in a prescribed (arbitrarily small) neighborhood of the fixed point.   
Correspondingly, most of the summands $L(u_{j-1},u_j)$ in \eqref{gn} are close to $L(1,1)=3$.

As important as the invariant curves are for understanding of the dynamics of the map $\Phi$, our proof of 
Shallit's asymptotics in its original form leaves them behind the scenes.
They will be used explicitly in the order-sharp analysis of Section~\ref{sec:convrates}. 

The points  of the trajectory $\Tra{n}$ will be denoted as $(p_j, u_j)$ when $n$ is fixed in the current context and as $(p_{j,n},u_{j,n})$
when a need arises to indicate the dependence of the coordinates on $n$ explicitly. 
The notation/convention $u_j=u_{j,n}$ has already been in use in \GDM.

\subsection{Identities} 
\label{ssec:conserv}

\begin{lemma}
\label{lem:conserv}
(a) If $(p',u')=\Phi(p,u)$, then $p'u'+u'=pu+p$.

(b) In particular, $p_j u_j+u_j=p_{j-1}u_{j-1}+p_{j-1}$ 
for any trajectory. 

(c) For the trajectory $\Tra{n}$ and $1\leq\ell\leq n$,  
$$
\sum_{j=1}^\ell (p_{j-1}-u_j)=p_\ell u_\ell.
$$
\end{lemma}

\begin{proof}
(a)  Simple check.

(c) It follows from (b) by telescopic summation, since $u_0=0$.  
\end{proof}

\subsection{Monotonicity with respect to the initial condition} 
\label{ssec:mon-init}
Let $P_0(t)=t$, $U_0(t)=0$ and define the functions $P_j(t)$, $U_j(t)$, $j=1,2,\dots$ recurrently by
$$
 (P_j(t),U_j(t))=\Phi(P_{j-1}(t),U_{j-1}(t)).
$$
Clearly, $P_j(t)$ and $U_j(t)$ are rational functions.

\begin{lemma}
\label{lem:incr}
\label{lem:uniq}

(a) The functions $P_{j}(\cdot)$ ($j\geq 0$) and $P_j(\cdot)U_j(\cdot)$ are increasing on $\RR_+$; the functions $U_{j}(\cdot)$ ($j\geq 1$) are decreasing.

(b) The trajectory $\{(p_j,u_j)\}=\{\Phi^{(j)}(p_0,0)\}$, $j=0,\dots,n$, with $p_{j-1},u_j>0$ ($j=1,\dots,n)$ and
$p_n=0$
is unique.

(c) The coordinates of the points on the trajectories $\Tra{n-1}$ and $\Tra{n}$ 
satisfy the inequalities 
\begin{equation}
\label{monot-in-n}
p_{j,n}>p_{j,n-1} \qquad (0\leq j\leq n-1).
\end{equation} 
\end{lemma}

\begin{proof} 
%
(a) 
It follows by induction using the identity of Lemma~\ref{lem:conserv}(a) in the form
$P_j U_j=P_{j-1} U_{j-1}+P_{j-1}-1/P_{j-1}$ 
and the trivial identity
$P_j=(P_j U_j) P_{j-1}$.

(b) For the given $n$, the initial coordinate $p_0$ satisfies the equation $P_n(p_0)=0$. By (a), a positive solution of this equation is unique. 

(c) Similarly, \eqref{monot-in-n} follows by (a) from the inequality 
$$
P_{n-j-1}(p_{j,n})=p_{n-1,n}>0=p_{n-1,n-1}=P_{n-j-1}(p_{j,n-1}).
$$ 
\end{proof}

{\bf Remark.}
The monotonicity property \eqref{monot-in-n} is experimentally observed in \GDM, Table 1 and Eqs.\ (9), (10). 

\subsection{Symmetry of the trajectory}
\label{ssec:sym}
 
\begin{lemma}
\label{lem:id}
(a) The map $\Phi$ is \emph{reversible}\ with respect to the involution $\sigma: (p,u)\mapsto(u,p)$,
that is, $\Phi^{-1}(p,u) =\sigma\circ\Phi(u,p)$.

(b) For the points on the trajectory for the given $n$, the identity
\begin{equation}
\label{revers}
u_j=p_{n-j} \qquad (0\leq j\leq n)
\end{equation}
holds. In particular, $u_n=p_0$.

(c) If $n=2k$, then $p_k=u_k$. If $n=2k-1$, then $p_{k-1}=u_{k}=1$. 
\end{lemma}

\begin{proof}
(a) Simple check.

(b) By (a), the boundary problem described in Sec.~\ref{ssec:dynsys} is invariant under the substitution $(p_{j},u_j)\mapsto (u_{n-j},p_{n-j})$.
 Due to the uniqueness result --- Lemma~\ref{lem:incr}(b) --- the equalities \eqref{revers} hold.

(c) For even $n$, the claim is a particular case of \eqref{revers}. For $n=2k-1$, we have $p_{k-1}=u_{k}$; on the other hand,
$u_{k}=1/p_{k-1}$ according to the definition of $\Phi$. Hence $u_k=1$. 
\end{proof}

{\bf Remarks.}
1. The identity $u_k=1$ in the case $n=2k-1$ is mentioned without elaboration in the first line of Eq.\ (7) in \GDM.

2. 
GDM also mention that the boundary conditions for the partial trajectory $\{(p_j,u_j)\}$, $0\leq j\leq k$,
\begin{equation}
\label{bcondk}
  u_0=0,   \qquad 
  p_k=\begin{cases}
        1, \quad \text{odd $n$},  \\
        u_k,\quad \text{even $n$},
  \end{cases}
\end{equation}
supply a more economical method for computing $p_0$ as compared to the method using the boundary conditions \eqref{bcond}. 

\subsection{Monotonicity of the trajectory}
\label{ssec:mon-traj}
 
\begin{lemma}
\label{lem:nw}
The trajectory $\Tra{n}$ propagates north-west. That is,
$$
  p_0>p_1>\dots>p_{n-1}>p_n=0; \qquad
  0=u_0<u_1<\dots<u_n.
$$
\end{lemma}

{\bf Remark.}
This property, as a numerical fact, is stated in \GDM, Eq.\ (7).

\begin{proof} 
Let $1\leq j\leq n$. We have $P_{n-j}(p_{j-1})=p_{n-1}>0$, while $P_{n-j}(p_{j})=p_{n}=0$. 
By Lemma~\ref{lem:incr}(a), $p_{j-1}>p_{j}$. The conclusion for the sequence $(u_j)$ follows trivially.  
\end{proof}

\begin{lemma}
\label{lem:subhyperbola}
The following inequalities hold along the trajectory $\Tra{n}$:
\\
 (a) $p_j u_j<1$  for $0\leq j\leq n$;
\\
(b)  $p_j\geq 1$ for $0\leq j\leq (n-1)/2$ and $u_j\geq 1$ for $(n+1)/2\leq j\leq n$;
\\
(c)  $u_j<1$ for $0\leq j\leq n/2$ and $p_j<1$ for $n/2\leq j\leq n$. 
\end{lemma}
 
\begin{proof}
(a) We have $p_j=1/u_{j+1}<1/u_j$ by Lemma~\ref{lem:nw}.

(b) 
%
If $0\leq j\leq (n-1)/2$, then $p_j^2\geq p_j p_{n-j-1}=p_j u_{j+1}=1$,
where we use Lemma~\ref{lem:nw} and Lemma~\ref{lem:id}(b). 

(c) is a consequence of (a) and (b).
\end{proof}

\subsection{Boundedness and limits} 
\label{ssec:pu-limits}

%
Put $\phi=(1+\sqrt{5})/{2}$.

\begin{lemma}
\label{lem:bdd}
The sequence $(p_{0,n})$ is bounded:  $p_{0,n}< \phi$ for all $n\geq 1$.
\end{lemma}

\begin{proof} Omitting the second index $n$, we have
$p_{0}>p_{1}$ by Lemma~\ref{lem:nw}. On the other hand,
$p_{1}=p_{0}^2-1$. Hence $p_{0}^2-p_{0}-1<0$, and the claim follows.
\end{proof}

From Lemmas~\ref{lem:nw}, \ref{lem:bdd} and the monotonicity \eqref{monot-in-n} it follows that for every fixed $j=0,1,2,\dots$ 
the sequence $(p_{j,n})$ ($n=1,2,\dots$) is increasing and bounded. Therefore  
there exist the limits
\begin{equation}
\label{limits-pu}
\ba{rcl}
  p_j^*&=&\uparrow\!\lim_{n\to\infty} p_{j,n}\qquad (j\geq 0),
  \\[1ex]
  u_j^*&=&\downarrow\!\lim_{n\to\infty} u_{j,n}=\dst\frac{1}{p_{j-1}^*} 
  \qquad (j\geq 1).
 \ea 
\end{equation}
Clearly, 
$\phi\geq p_0^*\geq p_1^*\geq p_2^*\geq \dots\geq 1$
and $1/p_0^*=u_1^*\leq u_2^*\leq\dots\leq 1$.

\medskip\noindent
{\bf Remarks.} 
1. The points $P_s$ and $P_u$ 
(Fig.~\ref{fig:invcurves}) have the coordinates $(p_0^*,0)$ and $(0,p_0^*)$, respectively. 
However, this intuitively obvious fact requires a proof (preceded by a precise definition of the curve $\gamma_s$), cf.\ 
 Lemma~\ref{lem:pt-on-invcurv} in Section~\ref{ssec:invcurves}.

\smallskip
2. Let us introduce the \emph{deviations from the fixed point}\
\begin{equation}
\label{dev}
\lambda_{j,n}=1-u_{j,n}.
\end{equation}
These parameters will be very useful in Section~\ref{sec:convrates}. 
Consistent with our convention, we will use the abbreaviated notation $\lambda_{j,n}=\lambda_j$ when $n$ is fixed
and denote the limit values $\lambda_j^*=1-u_j^*$. (Note that in \GDM, $\lambda_j$ stands for the limit values.)
Eliminating $p_j=p_{j,n}$ from the
relations that hold on the trajectory $\Tra{n}$,
$$
  p_j=\frac{1}{u_j+1}\qquad\mbox{\rm and}\qquad u_{j-1}=u_j^2(p_j+1)-1,
$$
yields the second order recurrence relation 
\begin{equation}
\label{rr-lambda}
\lambda_{j+1}=\frac{4\lambda_{j}-\lambda_{j-1}-2\lambda_{j}^2}{1+2\lambda_{j}-\lambda_{j-1}-\lambda_{j}^2}.
\end{equation} 
The limit $(n\to\infty)$ form of this recurence applies to the values $\lambda_j^*$ and coincides with Eq.\ (21) in \GDM. 



\subsection{Exponential proximity to the fixed point}
\label{ssec:exp-estim} 


\begin{lemma}
\label{lem:exp}
Put $\norm{j}=\min(j, n-j)$.
For $0\leq j\leq n$ the inequalities 
\begin{equation}
\label{exp-p}
|p_j-1|\leq 2^{-\norm{j}}
\end{equation}
and
\begin{equation}
\label{exp-pu}
0< 1-p_j u_j\leq \phi 2^{-\norm{j}}
\end{equation}
hold along the trajectory $\Tra{n}$. 
\end{lemma}

\begin{proof} 
To prove \eqref{exp-p}, consider three cases.

(i) $n/2\leq j\leq n-1$. Then $p_j u_j<1$ and $p_j<1$ 
by Lemma~\ref{lem:subhyperbola}(a,c). Hence 
$$
 1-p_{j+1}= 2-p_j^2(u_j+1)> 2-p_j-p_j^2=(1-p_j)(2+p_j)> 2(1-p_j)
$$
Since $1-p_n=1$, 
we conclude by induction 
that $1-p_j<2^{j-n}=2^{-\norm{j}}$.

(ii) $n$ odd, $j=(n-1)/2$. In this case $p_j-1=0$ by Lemma~\ref{lem:id}(c).

(iii) $0\leq j<(n-1)/2$. Then $1<p_j=u_{n-j}=1/p_{n-j-1}$
and $n-j-1\geq n/2$. By (i) we get
$$
 p_j-1=p_j(1-p_{n-j-1})<2^{-j-1}p_j.
$$
By Lemmas~\ref{lem:nw} and \ref{lem:bdd}, $p_j\leq p_0<\phi$. Since $\phi<2$, the claim follows. 

\smallskip
To prove \eqref{exp-pu}, due to symmetry 
we may assume that $j\leq n/2$, so that $1\leq p_j\leq \phi$
and $0\leq 1-u_j\leq 2^{-j}$ by \eqref{exp-p}.
We have
$$
  0< 1-p_j u_j=p_j(1-u_j)+(1-p_j)\leq p_j(1-u_j)\leq \phi 2^{-j},
$$
as required.
\end{proof}

\noindent
{\bf Corollary.}
$\lim\limits_{j\to\infty} p_j^*=\lim\limits_{j\to\infty} u_j^*=1$ and $\lim\limits_{j\to\infty}\lambda_j^*=0$.

\begin{proof}
Indeed, the estimates \eqref{exp-p} for $|1-p_{j,n}|$ are uniform in $n$ provided $n\geq 2j$.  
\end{proof}

{\bf Remark.} The Corollary justifies the formula (15) in \GDM.

\section{Convergence  of the sequence $(C_n)$}
\label{sec:asmin}

\subsection{Constants $C_n$ in terms of trajectory coordinates}
\label{ssec:Cn}

For a fixed $n\geq 1$, we assume, as before, that $\{(p_j,u_j)\}$, $j=0,1,\dots,n$, is the trajectory $\Tra{n}$ of the map $\Phi$
satisfying the boundary conditions \eqref{bcond}.

\begin{lemma}
\label{lem:Cn-traj}
Let $k=\lfloor n/2\rfloor$. The constant \eqref{Cn} can be expressed as follows,
\begin{equation}
\label{Cn-traj}
 C_n= 2\sum_{j=0}^{k-1} (3-2p_j-p_j u_j) +p_k^2.
\end{equation}
\end{lemma}

\begin{proof} Using Eq.~\eqref{gn} and the relations $1/u_j=p_{j-1}$, we express the constant \eqref{An}
as the value of $g_n(\uu^{(n)})$ on the trajectory $\Tra{n}$,
$$
  A_n=g_n(\uu)=\sum_{j=0}^n (p_j+u_j+p_j u_j).
$$
By symmetry (Lemma~\ref{lem:id}(b)) we have
$$
A_n=2\sum_{j=0}^{k} (p_j+u_j+p_j u_j)-\mu,
$$
where $\mu=0$ for $n$ odd and $\mu=p_k +u_k+p_k u_k$ 
for $n$ even.
Further, using Lemma~\ref{lem:conserv}(c)
and 
Lemma~\ref{lem:id}(c), we get
$$
 A_n=
 2\sum_{j=0}^{k-1} (2p_j+p_j u_j) +\mu'
$$
with $\mu'=2=3-p_k^2$ for $n$ odd and $\mu'=-p_k^2$ for $n$ even. The formula \eqref{Cn-traj} follows.
\end{proof}

\subsection{The existence of Shallit's constant}
\label{ssec:main}

Recall the limit values $p_j^*$, $u_j^*$ defined in \eqref{limits-pu}. Let
\begin{equation}
\label{SN}
 S_N=\sum_{j=0}^N (3-2p^*_j-p^*_j u^*_j).
\end{equation}


\begin{theorem}
\label{thm:Sh-lim}
There exist the limits
$
S=\lim_{N\to\infty} S_N
$
and
$
C=\lim_{n\to\infty} C_n
$.
The limit values are related by
$$
 C=2S+1.
$$
\end{theorem} 

\begin{proof}
Put $h_{j,n}=2(1-p_{j,n})+(1-p_{j,n}q_{j,n})$ for $j<n/2$ and $h_j^*=3-2p_j^*-p_j^*u_j^*$. 
By Lemma~\ref{lem:exp}, taking into account that $1-p_{j,n}<0<1-p_{j,n}q_{j,n}$ and $\phi<2$ we have $|h_{j,n}|<2^{1-j}$.
Passing to the limit as $n\to\infty$, we conclude that $|h_j^*|<2^{1-j}$, hence $\lim_{N\to\infty} S_N$ exists.

Let $k_n=\lfloor n/2\rfloor$. 
Put
$$
  z_{j,n}=\begin{cases}
  h_{j,n}-h_j^*
  \quad \mbox{if $0\leq j<k_n$},
  \\[1ex]
  0\quad \mbox{if $ j\geq k_n$}.
\end{cases}
$$
For every $j$ we have $\lim_{n\to\infty} z_{j,n}=0$ 
and the uniform bound 
$$
  \sum_{j=0}^\infty |z_{j,n}|\leq \sum_{j=0}^{k_n} |h_{j,n}|+\sum_{j=0}^\infty |h_j^*|\leq 8
$$
holds.
By the Dominated Convergence Theorem we obtain 
$$
 \lim_{n\to\infty}\left(\frac{C_n-1}{2}-S_{k_n-1}\right)=\lim_{n\to\infty} \sum_{j=0}^{k_n-1} z_{j,n}=0,
$$
which concludes the proof.
\end{proof}

\subsection{Monotonicity of the sequence $(C_n)$}
\label{ssec:Cn_monot}

Table 1 in \GDM\ clearly displays monotonicity of $C_n$. This property was not
used in the above proof of existence of Shallit's constant; we prove it now. 
 
\begin{prop}
\label{prop:Cn_incr}
The sequence $(C_n)$ is monotone increasing.
\end{prop} 

\begin{proof}
The vector minimizing $g_n(\uu)$ is $\uu^{(n)}=(0,u_{1,n},\dots,u_{n,n})$. As before, we fix $n$ and use shorter notation $u_k=u_{k,n}$;
the corresponding trajectory (Sec.~\ref{ssec:dynsys}) is $\Tra{n}=\{(p_n,u_n)\}$.
In order to show that $C_{n+1}>C_n$, it suffices to find a vector $\uu^\dag\in\RR^{n+2}$ such that 
$u^\dag_0=0$, $\hat\uu^\dag=(u^\dag_1,\dots,u^\dag_{n+1})\in\RR_+^{n+1}$ and $\delta=g_n(\uu^{(n)})+3-g_{n+1}(\uu^\dag)>0$.

Let $k\in\{1,\dots,n-1\}$ and $r>0$ be some parameters to be chosen later. Put
$$
 u^\dag_j=\begin{cases}
 u_j, \quad 1\leq j\leq k,
 \\
 r, \quad j=k+1,
 \\
 u_{j-1}, \quad k+2\leq j\leq n+1.
\end{cases}
$$
Then
$$
\delta=3+\frac{u_{k}}{u_{k+1}}-\frac{u_{k}}{r}-r-\frac{1}{r}-\frac{r}{u_{k+1}}
=3+p_{k}u_{k}-\frac{u_{k}+1}{r}-(p_{k}+1)r.
$$
The right-hand side is maximized at
$$
 r=\sqrt{\frac{1+u_{k}}{1+p_k}}.
$$
This choice of $r$ gives
$$
\delta
=
3+p_{k}u_{k}-2\sqrt{(1+p_{k})(1+u_{k})}.
$$
The inequality to prove, $\delta>0$, is equivalent to the inequality
$$
 (3+p_{k}u_{k})^2>4(1+p_{k})(1+u_{k}),
$$ 
which can be written as
$$
 (p_{k}u_{k}-1)^2+ 4(1-p_{k})(1-u_{k})>0.
$$
Now, if $n$ is even, then we put $k=n/2$, so that $p_{k}=u_{k}$. 
If $n$ is odd, then put $k=(n-1)/2$, so that $p_{k}=1$ and $u_{k}<1$ (Lemma~\ref{lem:id}).
In both cases the required inequality obviuosly holds.
\end{proof}

\medskip\noindent
{\bf Remark.}
The proof implies the estimates
\begin{equation}
\label{lowerbnd_difC}
  C_{n+1}-C_n>\begin{cases}
     (1-u_k)^2,\quad \mbox{\rm $n=2k$},\\
     (1-u_k)^2 \left(\frac18-o(1)\right),\quad \mbox{\rm $n=2k+1$}.
  \end{cases}
\end{equation}
Theorem~\ref{thm:convrates} in Section~\ref{sec:convrates} shows that the relation $C-C_n\asymp (1-u_k)^2$ is correct.

\section{Convergence rates}
\label{sec:convrates}

\subsection{Overview}
\label{ssec:convrates-intro}

In the previous sections we have proved the monotonicity properties of the sequences $(C_n)$ and $(p_{0,n})$ observed in Table~1 of GDM.
Examination of the quantitative experimental information from the same table reveals the asymptotic relations
$$
 p_0^*-p_{0,n}\asymp \rho^{-n},
 \qquad 
 C-C_n\asymp\rho^{-n}.
$$
In addition, Table~2 of GDM clearly suggests that $\lambda_j^*=1-u_j^*\asymp\rho^{-j}$.
These relations are proved here. 

Note that the inequalities in Lemma~\ref{lem:exp} lead to the estimates $ p_0^*-p_{0,n}=O(2^{-n})$ and $\lambda_j-\lambda_j^*=O(2^{-j})$.
With a little additional effort (approximating the map $\Phi$ by its linear part near the fixed point)
it would be not too difficult to replace the base $2$ by $\rho$ in these estimates. 
On the other hand, in order to prove any quantitative upper bound for the difference 
$C-C_n$ one needs to establish uniform estimates for $u_j^*-u_{j,n}$ and $p_j^*-p_{j,n}$ and this seems to be a much harder task. 

The powerful tool that we employ here to achieve all goals at once is Hartman's $C^1$ linearization theorem. In essense, it provides new coordinates
in which the map $\Phi$ becomes linear, while the distortion of distances between the old and new coordinate systems is bounded above and below.

The linearization easily yields all required upper as well as lower estimates for coordinate differences along the trajectories
provided the ``initial'' coordinate in the unstable direction is not vanishingly small.
The latter condition is ensured by the transversality properties proved in Section~\ref{ssec:transvers}.

The formula \eqref{Cn-traj} for $C_n$ used in the proof of Theorem~\ref{thm:Sh-lim} would yield the upper estimate $C-C_n=O(\rho^{-n/2})$.
The corresponding series \eqref{SN} has general ($j$-th) term that goes to zero as $O(\rho^{-j})$. GDM derived a different series representation
for $C$, see GDM, Eq.~(20), where the general term decays as $O(\rho^{-2j})$.  In eq.~\eqref{Cn-quad} of Section~\ref{ssec:Cn-quad} the same pattern is used
to  express the pre-limit constants $C_n$. That expression is good enoug to obtain the estimate $C-C_n=O(n\rho^{-n})$,
which still falls short of the experimental evidence. 
One more step towards convergence acceleration is made in Lemma~\ref{lem:Cn-cube}, where we obtain a representation for $C_n$
with general term that decays as $O(\rho^{-3j})$. (The convergence acceleration discussed in the last part of GDM is different and
does not suit our purposes.)  

After all these preparations, the main theorem is stated and proved in Section~\ref{ssec:convrate}.   

\subsection{Local linearization}
\label{ssec:linearization}

We refer to the theorem of P.~Hartman \cite{Hartman_1960} that asserts that a $C^2$ differeomorphism near a hyperbolic fixed point in two dimensions is $C^1$-conjugate to its linear part. 
(To appreciate a few subtleties around this result --- dimension, smoothness class, resonances --- in a more general context
see  \cite[Sec.~6.6]{KH_1996}.)

  
Specifically, for the our map $\Phi$  there exists a neighborhood
$\Omega$ of the fixed point $P_0=(1,1)$ and a $C^1$ diffeomorphism $h$ of $\Omega$ onto some neighborhood of $(0,0)$ 
such that $h(P_0)=(0,0)$ and
$$
 h\circ\Phi(p,u)=\Psi\circ h(p,u) \quad\mbox{\rm if $(p,u)\in\Omega\cap \Phi^{-1}\Omega$},
$$
where $\Psi$ is a linear map, which we identify with its matrix $\Psi=\mathrm{diag}(\rho^{-1},\rho)$.

Let $(\xi,\eta)=h(p,u)$ be the new coordinates. If $(\xi',\eta')=h(p',u')$, where $(p',u')=\Phi(p,u)$,
then $\xi'=\rho^{-1}\xi$, $\eta'=\rho\eta$. Thus, $\xi$ is the coordinate in the stable direction and $\eta$ --- in the unstable direction.

\medskip\noindent
{\bf Remark. }
A weaker form of linearization, such as, say, a $C^0$ linearization provided by the much widely applicable Grobman-Hartman theorem,
would not suffice for our purposes, since we need the linearizing local homeomorphism to be quasi-isometric, that is,
to change the distances by a factor that is uniformly bounded away from zero and infinity.  

\subsection{Boundary problem for trajectories of the linearized map}
\label{ssec:bvp-linear}

Lemma~\ref{lem:exp} guarantees that there exists $n_0\in\NN$ such that for $n/2\geq j\geq n_0$
we have $(p_{j,n},u_{j,n})\in\Omega$. Put 
$$
 (\xi_{j,n},\eta_{j,n})=h(p_{j,n},u_{j,n}),\qquad n_0\leq j\leq \lfloor n/2\rfloor.
$$
(We do not define and will not need numerical values $\xi_{j,n}$ and $\eta_{j,n}$ with $j\notin\{n_0,\dots,\lfloor n/2\rfloor\}$.)

Let us describe a boundary problem that uniquely determines the trajectory.

Denote by $\Gamma_1$ the image under the map $h\circ\Phi^{n_0}$ of an interval $I_1$ of the real line
containing the point $P_s$ so small that $\Gamma_1\subset h(\Omega)$. 
Denote by $\Gamma'_2$ and $\Gamma''_2$ the images under the map $h$ of intervals $I'_2$ and $I''_2$ of the lines,
respectively, $p=1$ and $p=u$, containing the point $(1,1)$ and so small that $\Gamma_2'\cup\Gamma''_2\subset h(\Omega)$. 

\begin{lemma}
\label{lem:traj-lin} Let $n> 2n_0$.
The point $Q=(\xi_{n_0,n},\eta_{n_0,n})\in h(\Omega)$ is uniquely determined by the following conditions:

 (i) $Q\in \Gamma_1$;
 
(ii) In the case of odd $n=2k+1$, $\Psi^{k-n_0}(Q)=(\xi\rho^{n_0-k},\eta\rho^{k-n_0})\in\Gamma'_2$, 
and in the case of even $n=2k$, $\Psi^{k-n_0}(Q)\in\Gamma''_2$.  
\end{lemma}

\begin{proof}
The initial point $Q^\#=(p_{0,n},u_{0,n})=(h\circ\Phi^{n_0})^{-1}(Q)$ of the of the trajectory $\{(p_{j,n},q_{j,n})\}$ is uniquely determined by
the conditions:

(i) $Q^\#$ lies on the semiaxis $u=0$, $p>0$;

(ii) In the case of odd $n=2k+1$, $\Psi^{k}(Q^\#)$ lies on the line $p=1$; 
in the case of even $n=2k$, $\Psi^{k}(Q^\#)$ lies on the line $p=u$. (Recall Lemma~\ref{lem:id}(c).)

The boundary conditions stated in the Lemma are equivalent to these in view of the definitions of the curves $\Gamma_1$, $\Gamma_2'$, $\Gamma_2''$. 
\end{proof}

\subsection{Invariant curves}
\label{ssec:invcurves}

We may assume for simplicity that the neighborhood $h(\Omega)$ of $(0,0)$ of Hartman's theorem is a square $|\xi|<\eps$, $|\eta|<\eps$.
The rectilinear interval $\gamma_s^\#=\{(\xi,\eta)\,|\,\eta=0, |\xi|<\eps\}$ is the stable invariant curve in the sense that 
$\Psi^j(Q)\to (0,0)$ as $j\to\infty$ for every $Q\in\gamma_s^\#$. Similarly, the interval $\gamma_u^\#=\{(\xi,\eta)\,|\,\xi=0, |\eta|<\eps\}$ is the 
unstable invariant curve in the sense that $\Psi^{-j}(Q)\to (0,0)$ as $j\to\infty$ for every $Q\in\gamma_u^\#$.

We {\it define}\ (a relevant part of) the invariant curve $\gamma_s$ in the $(p,u)$ plane as the image $(\Phi^{n_0}\circ h)^{-1}(\gamma_s^\#)$;
more precisely; we are only interested in that part of the image where $u>0$ and $p>0$. 
The notation $\gamma_s$ will henceforth apply to the said part. 

The following lemma is intuitively obvious; yet we state it explicitly.

\begin{lemma}
\label{lem:pt-on-invcurv}
The points $(p_j^*,u_j^*)$, $j=0,1,2,\dots$, defined in \eqref{limits-pu}  belong to $\gamma_s$.
\end{lemma}

\begin{proof}
By Lemma~\ref{lem:exp} and its corollary we have $(p^*_{j},u^*_{j})\to P_0$ as $j\to\infty$.
By continuity,  $(p^*_{j},u^*_{j})=\Phi^j(p^*_0,u^*_0)$, so every point $Q=h(p^*_n,u^*_n)$,
$j\geq n_0$, satisfies the characteristic property $\lim_{j\to\infty} \Psi^j (Q)=(0,0)$ of the points on the stable curve
in $(\xi,\eta)$ coordinates. The claim of Lemma follows.   
\end{proof}

\medskip\noindent
{\bf Remarks. }
1. By construction, the curve $\gamma_s$ is at least $C^1$. In particular, it has a tangent
at the point $P_s$ where $\gamma_s$ meets the $p$-axis. 

\smallskip
2. We will not need the other parts of the stable invariant curve in the $(p,u)$ coordinates nor the unstable curve.
The maximal continuous extensions of the invariant curves
can be defined by starting from the neigborhood $\Omega$ of the fixed point and iterating the map $\Phi$ or $\Phi^{-1}$ 
while the newly obtained parts of the curves lie in the quadrant $p>0$, $u>-1$ for the stable curve and $u>0$, $p>-1$ for the unstable curve, cf.\ Fig.~\ref{fig:invcurves}.

\smallskip
3. Alternatively, (the maximal continuous extensions of) the invariant curves can be defined as follows. Consider the stable curve;
the unstable one is obtained by the coordinate swap. Denote $D_0=\{(p,u)\,|\,p>0, u>-1\}$. Define recurrently
$D_n=\Phi^{-1}(D_{n-1})\cap D_0$.  For example, $D_1=\{(p,u)\,|\,p>0, u>p^{-2}-1\}$. Clearly, $D_0\supset D_1\supset D_2\supset\dots$.
Let $D_{\infty}=\cap_{n=1}^{\infty} D_n$. The boundary of $D_\infty$ is the stable invariant curve. (We omit the proof.)

\smallskip
4. It is visually obvious that the domains $D_n$ are convex. We prove this observation in the next section.
It can be skipped, since the result is not used in the sequel.

\subsection{Convexity of the domains $D_n$}
\label{ssec:conv_Dn}

\begin{prop}
\label{prop:conv_Dn}
  For every $n=0,1,2,\dots$ the domain $D_n$ defined in Remark 3 of Section~\ref{ssec:invcurves} is convex.
\end{prop}

\begin{proof}
Using the notation of Section~\ref{ssec:mon-init},
we will prove that $\ddot U_n(t) \dot P_n(t)-\dot U_n(t)\ddot P_n(t)>0$ if $n\geq 1$.
(The dot stands for $d/dt$.)
For $n=0$, the inequality is replaced by equality. 
Proceeding by induction, we have
$$
\ba{l}
 \dst
 \dot U_{n+1}=-\frac{\dot P_{n}}{P_n^2},
 \qquad \dst 
 \ddot U_{n+1}=\frac{2(\dot P_n)^2-P_n\ddot P_n}{P_n^3},
 \\[2.5ex]
 \dot P_{n+1}=P_n^2\dot U_n+2P_n(U_n+1)\dot P_n,
 \\[1ex]
 \ddot P_{n+1}=4 P_n \dot P_n\dot U_n+P_n^2\ddot U_n+2(U_n+1)(\dot P_n)^2+2P_n(U_n+1)\ddot P_n.
\ea
$$
Hence
$$
\ddot U_{n+1} \dot P_{n+1}-\dot U_{n+1}\ddot P_{n+1}
=
(\ddot U_n \dot P_n-\dot U_n\ddot P_n)
+\frac{6(\dot P_n)^2}{P_n^2} 
(P_n\dot U_n+(U_n+1)\dot P_n).
$$
By Lemma~\ref{lem:incr}(a), 
$(P_n U_n)\dot{\phantom{|}}+\dot P_n>0$,
so the induction step is complete.
\end{proof}

\subsection{Transversality lemma}
\label{ssec:transvers}

Recall (see Sec.~\ref{ssec:bvp-linear}) that $\gamma_s^\#$ denotes the stable curve in the $(\xi,\eta)$ coordinates and
$\Gamma_1$ is the curve containing the points $(\xi_{n,n_0},\eta_{n,n_0})$, which converge as $n\to\infty$ to $P_s^\#=(\xi^*_{n_0},0)$. 

\begin{lemma}
\label{lem:transversal}
The curves $\Gamma_1$ and $\gamma_s^\#$ meet at the point $P_s^\#$ transversally.
\end{lemma}

\begin{proof}
1. By applying the nondegenerate map $(h\circ \Phi^{n_0})^{-1}$, the problem reduces to showing that in the $(p,u)$-plane the curve
$\gamma_s$ meets the axis $u=0$ at $P_s=(p_0^*,0)$ transversally. 
Since $\gamma_s$ is a $C^1$ curve, it has a slope $\sigma$ at $P_0$ (possibly infinite).  
Our task is to show that $\sigma\neq 0$.

A general point of $\gamma_s$ near $P_s$ can be written in the form
$(p_0(t),u_0(t))=(p^*_0+\dot p_0 t+o(t), \dot u_0 t+o(t))$. Then $\sigma=\dot u_0/\dot p_0|_{t=0}$.

Let $(p_j(t),u_j(t))=\Phi^j(p_0(t),u_0(t))$. We omit the explicit argument $t$ in the rest of the proof.
The derivatives will always be evaluated at $t=0$. 

\smallskip
2. For large values of $j$ we have $\dot u_j/\dot p_j<0$, since the curve $\gamma_s$ approaches the point $P_s$
along the stable direction $(1,-\rho)$. To be definite, let us assume that $\dot p_j<0$ and $\dot u_j>0$.
Since $\rho>1$, we also have $(p_j u_j)\dot{\,}>0$.

By backward induction we will show that the same signs of the $t$-derivatives take place for all $j\geq 0$.
We use Lemma~\ref{lem:conserv}(a), Lemma~\ref{lem:id}(a), and argue similarly to the proof of Lemma~\ref{lem:incr}(a).
The induction step consists in proving that the conditions $\dot p_j<0$, $\dot u_j>0$, $(p_j u_j)\dot{\,}>0$
imply $\dot p_{j-1}<0$, $\dot u_{j-1}>0$ and $(p_{j-1} u_{j-1})\dot{\,}>0$. And indeed, this follows from the identities
$$
\begin{array}{c}\dst
 \dot p_{j-1}=-u_{j}^{-2} \dot u_{j},
 \qquad
 (p_{j-1}u_{j-1})\dot{\,}=(p_j u_j)\dot{\,}+\dot u_j-\dot p_{j-1},
 \\[1ex]\dst
 \dot u_{j-1}=u_j\left[(u_j p_j)\dot{\,}+\dot u_j(p_j+2)\right].
\end{array}
$$

\smallskip
3. Similarly to Lemma~\ref{lem:conserv} we have an identity for the trajectories used in this proof: for any $\ell\geq 1$
$$
  p_\ell u_\ell-p_0 u_0=\sum_{j=1}^\ell (p_{j-1}-u_j).
$$
Differentiating and setting $t=0$ we obtain, since $u_0(0)=0$,
\begin{equation}
\label{pudot}
 p_0^* \dot u_0=-\dot p_0+(p_\ell u_\ell)\dot{\,}+\sum_{j=1}^\ell (\dot u_j-\dot p_j).
\end{equation}
From the inequalities proved above we deduce
$$
 -\sigma=\frac{\dot u_0}{|\dot p_0|}> \frac{1}{p_0^*}.
$$
The lemma is proved.
\end{proof}

\medskip
\noindent
{\bf Remark. } 
Passing to the limit as $\ell\to\infty$ in the identity \eqref{pudot} and taking into account that $(p_\ell u_\ell)\dot{\,}\to 0$ 
(we leave the proof to the reader)  we obtain the formula
$$
  \sigma=-\frac{1}{p^*_0}\left(1+\sum_{j=1}^\infty \frac{\dot p_j-\dot u_j}{\dot p_0}\right)=-1.13\dots.
$$ 
This formula can be considered as an equation for $\sigma$ to be solved by iterations, 
given that the values $\dot p_0=-1$ (choice for convenience) and $\dot u_0=-\sigma$ determine all subsequent $\dot p_j$, $\dot u_j$.

\subsection{Sums for $C_n$ with rapidly decreasing terms}
\label{ssec:Cn-quad}

Recall the deviations $\lambda_{j,n}=1-u_{j,n}$ introduced in \eqref{dev}. We fix the trajectory $\Tra{n}$ and omit the index $n$ in the parameters $p_j$, $u_j$
and $\lambda_j$. 

\begin{lemma}
\label{lem:Cn-quad}
Let $k=\lfloor n/2\rfloor$. The expression \eqref{Cn-traj} for the constant $C_n$ can be transformed as follows:
\begin{equation}
\label{Cn-quad}
\begin{array}{rcl}
  C_n&=&
 \dst
 2\sum_{j=0}^{k-1} (3-p_j+u_j-2 u_{j+1}-p_j u_j)+\delta
\\[2ex]
   &=& 
  \dst 
  2\sum_{j=1}^{k} \frac{\lambda_{j}(\lambda_{j-1}-2\lambda_{j})}{u_{j}}+\delta,
\end{array}
\end{equation}
where $\delta=1-(p_k-1)^2$.
\end{lemma}

\begin{proof} 
The first line in the formula \eqref{Cn-quad} is obtained by adding together the identity \eqref{Cn-traj},
the telescopic identity
$$
 2\sum_{j=0}^{k-1} (u_j-u_{j+1})+2u_{k}=0,
$$
and the identity of Lemma~\ref{lem:conserv}(c) with $\ell=k$,
$$
 2\sum_{j=0}^{k-1} (p_j-u_{j+1})-2p_k u_{k}=0.
$$
The term outside the summation sign in the resulting expression is $\delta=p_k^2-2p_k u_k +2 u_k$. 
If $n$ is even, then $u_k=p_k$, and if $n$ is odd, then $p_k=1$; in both cases the formula $\delta=1-(p_k-1)^2$ is valid.

\smallskip
To obtain the  second  line in \eqref{Cn-quad}, we substitute $p_j=1/u_{j+1}$, which yields 
$$
 3-\frac{1}{u_{j+1}}+u_j-2u_{j+1}-\frac{u_j}{u_{j+1}}=\frac{(u_{j+1}-1)(u_j-2u_{j+1}+1)}{u_{j+1}}.
$$
In the final expression the index shift $j+1\mapsto j$ is made. 
%
\end{proof}

The formula \eqref{Cn-traj} written in terms of the deviations becomes
$$
  C_n=2\sum_{j=1}^{k} \frac{\lambda_{j-1}-3\lambda_{j}}{1-\lambda_{j}}+p_k^2.
  \eqno(\ref{Cn-traj}')
$$
We think of $\lambda_{j}$ as of small quantities (for large $j$), and here the terms decay linearly in $\lambda_{j}$, while in
in the formula \eqref{Cn-quad} they decay as $\lambda_{j}^2$. In the proof of Theorem~\ref{thm:convrates} even this rate of decay will not quite
suffice; we need a cubic decay in $\lambda_{j}$.  We will use the appropriate version of the $O$-notation in order to avoid cumbersome 
explicit formulas in the next lemma and its proof.

If $f(x_1,x_2,x_3)$ is a rational function such that 
$f(x_1,x_2,x_3)\leq \const\cdot \max_{i=1,2,3}|x|_i^d$ for $0\leq x_1\leq 1$, $0\leq x_2\leq \lambda_1^*$ and $0\leq x_3\leq \lambda_2^*$,    
then we will write $f(\vec{\lambda_j})=O(\lambda_{j\pm})$, where $\vec{\lambda_j}=(\lambda_{j-1},\lambda_{j},\lambda_{j+1})$.
(Note that $\lambda_{j,n}\leq \lambda^*_j$ for all $n$, cf.~\eqref{limits-pu}.)

The general summand in the formula (\ref{Cn-traj}$'$) is obviously $O(\lambda_{j\pm})$ and in the formula \eqref{Cn-quad} (second line), it is $O(\lambda_{j\pm}^2)$.
(In both cases the dependence on $\lambda_{j+1}$ is vacuous.)


\begin{lemma}
\label{lem:Cn-cube}
There exist functions $f$ and $r$ of three arguments such that
$f(\vec{\lambda_j})=O(\lambda_{j\pm}^3)$, $r(\vec{\lambda_j})=O(\lambda_{j\pm}^2)$,
and 
$$
 C_n=\sum_{j=1}^k f(\vec{\lambda_j})+1+r(\vec{\lambda_k}),
$$
where $k=\lfloor n/2\rfloor$ and the
values $\lambda_j=\lambda_{j,n}$ are taken 
on the trajectory $\Tra{n}$. 
\end{lemma}

\begin{proof} 
The recurrence relation \eqref{rr-lambda} implies that
$$
  \lambda_{j-1}+ \lambda_{j+1}=4\lambda_j+O(\lambda_{j\pm}^2).
$$
From this we deduce
$$
\ba{l}
\dst
 2\sum_{j=1}^{k} \lambda_{j-1}\lambda_j=\sum_{j=1}^{k} \lambda_{j-1}\lambda_j+\sum_{j=0}^{k-1} \lambda_j\lambda_{j+1} 
=\lambda_0\lambda_1+\sum_{j=1}^{k} (\lambda_{j-1}+\lambda_{j+1})\lambda_j-\lambda_k\lambda_{k+1}
\\[3ex]
\dst \qquad\qquad\qquad
=\sum_{j=1}^{k} (4\lambda_{j}^2+O(\lambda_{j\pm}^3))- O(\lambda_{k\pm}^2).
\ea
$$
Using the fact that $1/u_j=1+O(\lambda_{j\pm})$ we may write the formula \eqref{Cn-quad} in the form
$$
 C_n=2\sum_{j=1}^k \left(\lambda_{j-1} \lambda_{j}-2\lambda_{j}^2+O(\lambda_{j\pm}^3)\right)+1+O(\lambda_{k\pm}^2).
$$
Substituting the previous identity to the latter one we come to the claimed formula. 
\end{proof}

\noindent
{\bf Remark. } A rather tedious calculation that was replaced by the $O$-estimates in the proof of Lemma~\ref{lem:Cn-cube} 
leads to the following formula for $C$:
\begin{equation}
\label{C-cubic}
C=p^*_0+\sum_{j=1}^\infty \frac{{\lambda^*_j}^2(\lambda^*_{j+1}-2\lambda^*_j+\lambda^*_j\lambda^*_{j+1})}{u^*_j u^*_{j+1}},
\end{equation}
where the general term of the series decays as $O(\rho^{-3j})$ and, due to misterious cancelations, does not depend on $\lambda_{j-1}^*$. 

\subsection{Main theorem on convergence rates}
\label{ssec:convrate}

\begin{theorem}
\label{thm:convrates}

The following asymptotic equivalences hold. (The implied constants are independent of $n$ and $j$.)

(a) The distance from the fixed point: for $0\leq j\leq n/2$,
\begin{equation}
\label{conv-p1}
|p_{j,n}-1|+|u_{j,n}-1| \asymp \rho^{-j}
\end{equation}
and for $k=\lfloor n/2\rfloor$,
\begin{equation}
\label{lowerbnd_difUk1}
 |u_{k,n}-1|\asymp \rho^{-n/2}.
\end{equation}
(b)The distance to the limit trajectory: also for  $0\leq j\leq n/2$,
\begin{equation}
\label{conv-pp}
|p_{j,n}-p_j^*|+|u_{j,n}-u_j^*|\asymp\rho^{j-n}.
\end{equation}
In particular,
$$
 |p_{0,n}-p_0^*|\asymp \rho^{-n}.
$$
(c) Convergence rate of the sequence $(C_n)$: 
\begin{equation}
\label{conv-C}
C-C_n\asymp\rho^{-n}.
\end{equation}

\end{theorem}

\begin{proof}
Note first that it is enough to prove \eqref{conv-p1} and \eqref{conv-pp} for $j\geq n_0$ and assuming that $n>2n_0$, where $n_0$ is 
determined by the size of the neighborhood $\Omega$ provided by Hartman's theorem (Section~\ref{ssec:bvp-linear}). 

Then, since the linearizing diffeomorphism $h$ is $C^1$, it is enough to prove the corresponding estimates in the coordinates
$(\xi,\eta)$ where the dynamics is linear. (Cf.\ Remark in Section~\ref{ssec:linearization}.)

Throughout the proof we denote $k=\lfloor n/2\rfloor$.

\smallskip
(a) Put $d_{j,n}=|\xi_{j,n}|+|\eta_{j,n}|$. 
By Lemma~\ref{lem:traj-lin} we have 
$$
 d_{j,n}=\rho^{k-j}|\xi_{k,n}|+\rho^{j-k}|\eta_{k,n}|. 
$$
The curves $\Gamma_2'$ and $\Gamma_2''$ defined in 
Section~\ref{ssec:bvp-linear} are transversal to the axes of the $(\xi,\eta)$ coordinate plane, since their images under $h^{-1}$ are transversal
to the stable and unstable directions of the map $\Phi$ at the fixed point $P_0$.
Therefore 
$$
 m_1\rho^{k-j} d_{k,n}< d_{j,n} <M_1\rho^{k-j} d_{k,n}
$$
for appropriate constants $0<m_1<M_1$ (independent of $n$) and $n_0\leq j\leq k$.
On the other hand, $m_1'<d_{n_0,n}<M_1'$ for some constants $0<m_1'<M_1'$.
It follows that
$$
 m_1''\rho^{-j} < d_{j,n} <M_1''\rho^{-j} 
$$  
with
$$
 m_1''=\frac{m_1 m_1'}{M_1}\rho^{n_0},
\qquad
 M_1''=\frac{M_1 M_1'}{m_1}\rho^{n_0}.
$$
The estimate \eqref{conv-p1} is proved.

The estimate  \eqref{lowerbnd_difUk1} is due to the fact that $|p_k-1|\leq |u_k-1|$ for all $n$. (For even $n$, we have equality and for odd $n$
the left-hand side is zero.)

\smallskip
(b) Put $d^*_{j,n}=|\xi_{j,n}-\xi^*_{j}|+|\eta_{j,n}|$. (Recall --- see Section~\ref{ssec:bvp-linear} --- that $(\xi^*_j,0)=h(p^*_j,u^*j)\in\gamma_s^\#$.)  
By linearity of the map $\Psi$ we have
$$
 d^*_{j,n}=\rho^{n_0-j}|\xi_{n_0,n}-\xi^*_{n_0}|+\rho^{j-n_0}|\eta_{n_0,n}|. 
$$
Since $(\xi_{n_0,n},\eta_{n_0,n})\in \Gamma_1$ (Lemma~\ref{lem:traj-lin}), the transversality of $\Gamma_1$ and $\gamma_s^\#$
(Lemma~\ref{lem:transversal}) implies that the ratio $|\xi_{n_0,n}-\xi^*_{n_0}|/|\eta_{n_0,n}|$ is bounded between two positive constants independent of $n$.
Therefore
$$
 m_2\rho^{j-n_0} d^*_{n_0,n}< d^*_{j,n} <M_2\rho^{j-n_0} d^*_{n_0,n}
$$
for appropriate constants $0<m_2<M_2$ (independent of $n$) and $n_0\leq j\leq k$.

As in proof of (a), we invoke the fact that $\Gamma_2'$ and $\Gamma_2''$ are transversal to $\gamma_s^\#$.
Using the estimates of part (a) and the estimate
$|\xi_k^*|=O(\rho^{-k})$, which easily follows from (a), we conclude that 
$$
  m_2' \rho^{-k}\leq d^*_{k,n}\leq M_2' \rho^{-k}
$$
for appropriate constants $0<m_2<M_2'$.
Therefore 
$$
m_2''\rho^{j-2k} < d^*_{j,n} <M_2''\rho^{j-2k} 
$$
with
$$
 m_2''=\frac{m_2 m_2'}{M_2},
\qquad
 M_2''=\frac{M_2 M_2'}{m_2}.
$$
Since $n-1\leq 2k\leq n$, the estimates as required are obtained.

Note that in the proof of part (a) the factors $\rho^{n_0}$ were justly included in the constants ($m_1''$ and $M_1''$) because
$n_0$ is a constant. Here in the proof of part (b) we cannot include the factors $\rho^{-2k}$ in the constants ($m_2''$ and $M_2''$),
since $k$ depends on $n$.

\smallskip
(c) The lower bound follows from Eqs.\ \eqref{lowerbnd_difC} and \eqref{lowerbnd_difUk1}. It remains to prove that $C-C_n=O(\rho^{-n})$.
Let $f$ be the function from Lemma~\ref{lem:Cn-cube}.
Introduce (compare \eqref{SN})
$$
 \tilde S_N=\sum_{j=1}^N
 f(\vec{\lambda}_j^*).
$$
By part (a), $\lambda_j^*=O(\rho^{-j})$, hence there exists $S=\lim_{N\to\infty} \tilde S_N$.
(We use the same notation for the limit as in Theorem~\ref{thm:Sh-lim}; the equality of the two will immediately follow from the proof.)

By Lemma~\ref{lem:Cn-cube} we have 
$$
 C-C_n=2(S-\tilde S_k)+2\sum_{j=1}^k  \left(f(\vec{\lambda}^*_j)-f(\vec{\lambda}_{j,n})\right)+
 \left(r(\vec{\lambda}^*_k)-r(\vec{\lambda}_{k,n})\right).
$$
Let us estimate the right-hand side. 

 1) Since $f(\vec{\lambda_j^*})=O(\lambda_{j\pm}^3)$, we have $S-\tilde S_k=O(\rho^{-3k})=O(\rho^{-3n/2})$.

 2) Since $r$ is a rational function and $r(\vec\lambda_{j})=O(\lambda_{j\pm}^2)$, 
  we have $r(\vec\lambda^*_k)-r(\vec\lambda_{k})=O(\lambda_{k\pm})\|\vec\lambda^*_k-\vec\lambda_{k}\|$.
  By part (a), $O(\lambda_{k\pm})=O(\rho^{-n/2})$, and by part (b) also $\|\vec\lambda^*_k-\vec\lambda_{k}\|=O(\rho^{-n/2})$.
  Thus $r(\vec\lambda^*_k)-r(\vec\lambda_{k})=O(\rho^{-n})$.
  
 3) Similarly we have
 $$
 f(\vec\lambda^*_j)-f(\vec\lambda_{j})=O(\lambda_{j\pm}^2)\|\vec\lambda^*_j-\vec\lambda_{j}\|=O(\rho^{-2j})\,O(\rho^{j-n})=O(\rho^{-j-n}). 
 $$
 Summing from $j=1$ to $k$ gives the estimate $O(\rho^{-n})$ for the sum.%
 \footnote{It is here that using Lemma~\ref{lem:Cn-quad} instead of Lemma~\ref{lem:Cn-cube} would lead to the weaker estimate $C-C_n=O(n\rho^{-n})$ with a superfluous linear factor.} 
This concludes the proof.
\end{proof}

\section{Remarks on the dynamics of the map $\Phi$}
\label{sec:remarks}

The dynamical system determined by the map \ref{mapPhi} can be an object of interest in its own right.  
Let us draw the attention to some points that may prompt generalizations or further exploration. 

\begin{enumerate}
\item 
From the point of view of symplectic dynamics with discrete time, the function $L(\cdot,\cdot)$ in Eq.~\eqref{gn} is interpreted as Lagrange's function 
and the function $g(\uu)$ --- as the action, cf.\ e.g.\ \cite{MDS_1998}, end of Section~9.1.
The trajectories satisfying specific boundary conditions arise from the principle of least action. 

\item 
For the purposes of this paper, only the positive quadrant (indeed, only the curvilinear quadrangle $0P_s P_0 P_u$, see Fig.~\ref{fig:invcurves}) is relevant
as regards the domain of the map $\Phi$. Going beyond the positive quadrant, there is another (elliptic) fixed point, $(-1,-1)$, and a continuous family of 4-cycles of the form $\{(-1,t), (t,-1),(-1,1/t),(1/t,-1)\}$ but no other obvious ``regular'' trajectories.
The extension of the invariant curves $\gamma_s$ and $\gamma_u$ beyond
the positive quadrant leads to the global invariant curves consisting of countably many smooth segments punctured at the points with $p=0$.
It would be interesting, in particular, to find other families of periodic trajectories and to characterize the set of non-wandering points of the system. 

\item
In connection with the two items above, one may ask about a meaningful way to introduce an additional parameter into the problem 
so as to obtain an interesting family, in which parameter-dependent phenomena (such as genericity of behaviour and bifurcations) can be studied. 

\item
Strictly speaking, since $\Phi(p,u)$ is not defined when $p=0$,
the global dynamics of the map $\Phi$ is defined on the set $\RR^2\setminus\cup_{j=0}^\infty (\Phi^{j}(X_1)\cup\Phi^{-j}(X_2))$,
where $X_1=\RR\times\{0\}$ and $X_2=\{0\}\times \RR$ are the coordinate axes. 
Can the action be extended in a natural way to a connected (and compact?) phase space by adjoining limit points corresponding to $\Phi(0,\cdot)$
and $\Phi^{-1}(\cdot,0)$?
\end{enumerate}

\section*{Appendix: Numerical constants}

For reference, we give below the numerical values of the constant $p_0^*$, see Sec.~\ref{sec:trajan}{\bf D}, with 400 decimals 
and of Shallit's constant $C$ with 401 decimals (so that round-down preserves the last decimal).   
The analysis of errors based on the results of Section~\ref{sec:convrates} leads to the following computational recommendations:

\begin{enumerate}
\item
  The accuracy of computations should be about the same as the desired accuracy of the results.
In particular, the presented results were obtained using CAS Maple (ver.14) with accuracy parameter \verb! Digits:=406!.   
Further increase of the accuracy had no effect on the presented decimals. 

\item
  In order to compute $p_0^*$ to the accuracy $\eps$, one can approximate it by $p_{0,n}$ with $n\approx -\log(\rho)/\log(\eps)$.
  In our case, $400\log(10)/\log(\rho)\approx 700$, and $n=702$ was sufficient to obtain all the presented decimals.  

\item
 In order to compute $C$ to the accuracy $\eps$, one needs to know $p_0\approx p_0^*$ to the accuracy of about $\eps$
 and to compute the partial sum of the series \eqref{C-cubic} with about $N\approx (2/3)\log\eps/\log\rho$ terms.
 (The terms are computed by iterating the map $\Phi$ starting with $(p_0,0)$.) In our case, $470=(2/3)\cdot 710$ terms sufficed. 
\end{enumerate}




To compute $p_{0,n}$ we used the shooting method with varying initial value $p_0$ and the terminal condition \eqref{bcondk}.   
The method of bisections was used to evaluate $p_0$ to the desired accuracy.

The results:

{\small
$$
\ba{rl}
 C=&
1.36945140399377005843552792420621433660771875900631 \\ &
\;\;87665783890080147149175646469894434570927342684376 \\ &
\;\;34144005789481431365880002586689693733190308998890  \\ &
\;\;87661338724042220421629124855821828996392195797332  \\ &
\;\;37120786480772194060018711007212918114161859594878  \\ &
\;\;04747713203416025094719841701277551146944176869331 \\ &
\;\;22641568691652661120042454933291650324779877238620 \\ &
\;\;756313168644067581730655070193831898528418301296696
\ea
$$

$$
\ba{rl}
p_0^*=&
1.44705435001627940656436532022322150134511477660996 \\ &
\;\;33541911604260928884594955381538985337173235890178 \\ &
\;\;44526143413324403274382574686028805322113073504874 \\ &
\;\;00334595332938142346550419137468567444603348994551 \\ &
\;\;35796272850688980015659307375350206718027627632733 \\ &
\;\;42268003719961619375942126945431930724800205584648 \\ &
\;\;72216579711992054958880069053860364912122611655716 \\ &
\;\;63216645295020299203349516473157637104275782708157 
\ea
$$
%

}

\end{document}